\documentclass[12pt]{amsart}
\usepackage{latexsym, url}
\usepackage{amsthm}
\usepackage{amsmath}
\usepackage{amsfonts}
\usepackage{amssymb}
\usepackage[dvips]{graphicx}
\usepackage{xypic}
\input xy
\xyoption{all}

\newtheorem{theo}{Theorem}[section]
\newtheorem{lemma}[theo]{Lemma}

\newtheorem{obs}[theo]{Observation}
\newtheorem{propo}[theo]{Proposition}
\newtheorem{defi}[theo]{Definition}
\newtheorem{coro}[theo]{Corollary}
\newtheorem{rem}[theo]{Remark}

\newtheorem{nota}[theo]{Notation}
\newtheorem{exam}[theo]{Example}
\newtheorem{exams}[theo]{Examples}

\newcommand\Alg{\operatorname{\bf Alg}}
\newcommand\Mon{\operatorname{\bf Mon}}

\newcommand\op{\operatorname{op}}
\newcommand\ob{\operatorname{ob}}
\newcommand\id{\operatorname{id}}

\newcommand\Set{\operatorname{\bf Set}}
\newcommand\Pos{\operatorname{\bf Pos}}
\newcommand\Met{\operatorname{\bf Met}}
\newcommand\PMet{\operatorname{\bf PMet}}

\newcommand\Ban{\operatorname{\bf Ban}}
\newcommand\Norm{\operatorname{\bf Norm}}

\newcommand\CMet{\operatorname{\bf CMet}}

\newcommand\Ins{\operatorname{Ins}}

\newcommand\colim{\operatorname{colim}}

\newcommand\ca{\mathcal {A}}

\newcommand\cd{\mathcal {D}}

\newcommand\cg{\mathcal {G}}

\newcommand\ce{\mathcal {E}}

\newcommand\ck{\mathcal {K}}

\newcommand\cm{\mathcal {M}}

\newcommand\ct{\mathcal {T}}

\newcommand\cv{\mathcal {V}}

\newcommand\eps{\varepsilon}
\newcommand\pa{\parallel}

 \newbox\noforkbox \newdimen\forklinewidth
\forklinewidth=0.3pt \setbox0\hbox{$\textstyle\smile$}
\setbox1\hbox to \wd0{\hfil\vrule width \forklinewidth depth-2pt
 height 10pt \hfil}
\wd1=0 cm \setbox\noforkbox\hbox{\lower 2pt\box1\lower
2pt\box0\relax}

\date{July 25, 2021}
 
\begin{document}
\title[Metric monads]
{Metric monads}
\author[J. Rosick\'{y}]
{J. Rosick\'{y}}
\thanks{Supported by the Grant Agency of the Czech Republic under the grant 
               19-00902S} 
\address{
\newline J. Rosick\'{y}\newline
Department of Mathematics and Statistics,\newline
Masaryk University, Faculty of Sciences,\newline
Kotl\'{a}\v{r}sk\'{a} 2, 611 37 Brno,\newline
Czech Republic}
\email{rosicky@math.muni.cz}

\begin{abstract}
We develop universal algebra over an enriched category $\ck$ and relate it to finitary enriched monads over $\ck$. Using it, we deduce recent results about ordered universal algebra where inequations are used instead of equations. Then we apply it to metric universal algebra where
quantitative equations are used instead of equations. This contributes to understanding of finitary
monads on the category of metric spaces.
\end{abstract}

\keywords{equational theory, monad, enriched category, metric space, poset}
\maketitle

\section{Introduction}
Classical universal algebra specifies algebras by means of finitary operations and equations. Since Bloom \cite{Bl}, ordered universal algebra, i.e., universal algebra over
posets, uses inequations $\leq$ instead of equations.
Our aim is to get a better understanding of metric universal algebra, i.e., universal algebra over metric spaces. Motivated by probabilistic programming, metric universal algebra was introduced in Mardare, Panangaden and Plotkin \cite{MPP} where quantitative equations $=_\eps$ were used instead of equations. We will show that both cases are instances of a categorical universal algebra where only operations and equations are used.

To explain this, one has to start with Lawvere theories \cite{La} using $(X,Y)$-ary operations $\omega$ interpreted on a set $A$ as mappings $\omega_A:A^X\to A^Y$. Since $Y$ is the coproduct $Y\cdot 1$ of $Y$ copies of the one-element set $1$, these operations are nothing else that $Y$-tuples of $X$-ary operations. Over posets, one uses finite posets as arities and interprets $A^X$ either as a set of monotone mappings $X\to A$ or as a poset of these monotone mappings. Here, $Y$ is a colimit of copies of the terminal poset $1$ and the two-element chain $2$, which replaces equations of $(X,Y)$-ary operations by inequations of $X$-ary operations. When $A^X$ is taken as a poset and operations $\omega_A$ are monotone, the resulting algebras are called coherent in Ad\' amek et al \cite{AFMS}.
Similarly, over metric spaces, one uses finite metric spaces as arities and, for a metric space $A$, interprets $A^X$ as the set of nonexpanding mappings $X\to A$. Now, $Y$ is a colimit of copies of the terminal space $1$ and the space $2_\eps$ having two points of distance $\eps$, which replaces equations of $(X,Y)$-ary operations by quantitative equations of $X$-ary operations. This corresponds to metric algebras of Weaver \cite{W}. To get the quantitative algebras of \cite{MPP}, one has to interpret $A^X$ as the metric space of nonexpanding mappings $X\to A$ and take $\omega_A:A^X\to A^Y$ as nonexpanding.  

The idea of using $(X,Y)$-ary operations where arities are objects $X$ and $Y$ of a category $\ck$ is due to Linton \cite{L} who showed that the category of algebras for a monad $T:\ck\to\ck$ can be given by equations between such operations. Here, $A^X$ is the set $\ck(X,A)$. The enriched case, where $\cv$ is a symmetric monoidal closed category, $\ck$ is a $\cv$-category and $A^X$ is the object $\ck(X,A)$ of $\cv$, was done in Dubuc  
 \cite{D} where $\cv$-equational theories over a $\cv$-category were introduced. 

To do universal algebra over a category $\ck$ one needs a concept of finiteness in $\ck$ to be able to speak about finitary operations. The optimal case is when $\ck$ is locally finitely presentable and finite arities are finitely presentable objects in $\ck$. In the enriched case, one needs that $\cv$ is locally finitely presentable as a closed category and $\ck$ is a locally finitely presentable $\cv$-category (see Kelly \cite{K1}). Since the category $\Pos$ of posets with monotone mappings is locally finitely presentable and cartesian closed, we do not have any problem with finiteness here. But the category $\Met$ of metric spaces (distance $\infty$ is allowed because, otherwise, the category would not be complete, or cocomplete) and nonexpanding mappings is symmetric monoidal closed and only locally $\aleph_1$-presentable. So, finite presentability cannot be taken as finiteness, which makes metric universal algebra much more difficult. But finite metric spaces can be characterized as those which are finitely generated w.r.t. isometries. This idea goes back to Gabriel and Ulmer \cite{GU} and was further developed in \cite{AR2} and \cite{DR}. Let us add that the importance of allowing distance $\infty$ was stressed in Lawvere \cite{La1}.
 
Classical universal algebra can be also described by finitary monads $T$ on the category $\Set$ of sets. Here, a monad is finitary if $T$ preserves filtered colimits. Indeed, 
Linton \cite{L1} showed that finitary monads on $\Set$ correspond to Lawvere theories. For a symmetric monoidal closed category $\cv$ which is locally finitely presentable as a closed category, Power \cite{P} introduced enriched Lawvere theories and showed that they correspond to enriched finitary monads on $\cv$. In Nishizawa and Power \cite{NP}, this was extended to finitary monads on a locally finitely presentable $\cv$-category. This immediately generalizes to every regular cardinal $\lambda$ and describes
$\lambda$-ary enriched monads, i.e., enriched monads preserving $\lambda$-filtered colimits, on every locally $\lambda$-presentable $\cv$-category $\ck$ provided that $\cv$ is locally $\lambda$-presentable as a closed category. The resulting enriched Lawvere theories from \cite{NP} do not fit our aim to do universal algebra over a $\cv$-category $\ck$. Recently, Bourke and Garner \cite{BG} introduced $\ca$-pretheories for every small dense subcategory $\ca$ of $\ck$ and related them to $\ca$-nervous enriched monads on $\ck$. Their pretheories perfectly suit our needs and describe $\lambda$-ary monads on every locally $\lambda$-presentable $\cv$-category $\ck$ provided that $\cv$ is locally $\lambda$-presentable as a closed category. Indeed, if $\ca$ is the full subcategory of $\lambda$-presentable objects of $\ck$ the $\ca$-pretheory of \cite{BG} is given by $(X,Y)$-ary operations and equations; here $X$ and $Y$ are $\lambda$-presentable objects of $\ck$. Over $\Met$, we can apply \cite{DR} and relate equational theories whose operations have finite metric spaces as arities to monads preserving directed colimits of isometries.

Both ordered algebras and quantitative algebras form computationally relevant structures
making possible an algebraic approach to the semantics of programming languages
and computational effects. Since Moggi \cite{M}, monads play an important role here
(see the survey Hyland and Power \cite{HP} and the overview Plotkin and Power \cite{PP}). 
Add that an important paper relating enriched monads and theories is Lucyshyn-Wright \cite{LW1}.

For the comfort of (some) readers, we will do the unenriched case at first. In Section 3, we show that $\lambda$-ary pretheories of \cite{BG} can be presented as equational theories over a category $\ck$ introduced in \cite{R} (motivated by \cite{L}). We show how it can be used to prove the result from \cite{BG} that $\lambda$-ary pretheories over a locally $\lambda$-presentable category $\ck$ correspond to $\lambda$-ary monads on $\ck$. Then we apply it to the case when $\ck$ is not locally finitely presentable but only locally finitely generated in the sense of \cite{DR}, which is the case of metric spaces. In Section 4, we do the enriched case and show that it yields some recent results from 
Ad\' amek et al \cite{AFMS} and \cite{ADV}. In Section 5, we discuss monads on metric spaces in more detail. In particular, we show that monads given by finitary unconditional equational theories of Mardare et al \cite{MPP1} preserve not only filtered colimits but also sifted ones. The same property have monads given by some $\otimes$-theories, for instance the monad given by normed spaces. But a characterization of monads on $\Met$ preserving filtered (or sifted) colimits remains unknown.

\vskip 1mm
\noindent
{\bf Acknowledgement.} The author is grateful to Ji\v r\'\i{} Ad\' amek for inspiring discussions, to John Bourke for a feedback concerning pretheories and to the referees
for their valuable comments.

\section{Preliminaries about categories and metric spaces}
All needed facts about locally presentable categories can be found in \cite{AR},\cite{GU}, or \cite{MP}. In particular, we will freely use that fact that $\lambda$-filtered colimits can be substituted by $\lambda$-directed ones (see \cite[Remark 1.21]{AR}). A monad on a locally presentable category will be called \textit{$\lambda$-ary} if it preserves $\lambda$-filtered colimits, which is equivalent to the preservation of $\lambda$-directed colimits. We will also need $\cm$-locally generated categories $\ck$ where $(\ce,\cm)$ is a $\lambda$-convenient factorization system on $\ck$ (see \cite{DR}).
Here, an object $A$ of $\ck$ is $\lambda$-generated w.r.t. $\cm$ if $\ck(A,-):\ck\to\Set$
preserves $\lambda$-directed colimits of morphisms from $\cm$. Now, a cocomplete category $\ck$ is $\cm$-locally $\lambda$-presentable if it has a set $\ca$ of $\lambda$-generated objects w.r.t. $\cm$ such that every object of $\ck$ is a $\lambda$-directed colimit of objects from $\ca$ and morphisms from $\cm$. Every $\cm$-locally $\lambda$-generated category is locally $\mu$-presentable for some $\mu\geq\lambda$. But it does not need to be locally $\lambda$-presentable.

All needed facts about enriched categories can be found in \cite{Bo}, or \cite{K}. Like 
in \cite{K1}, $\lambda$-filtered colimits are conical ones. \cite{K1} is also a reference for locally presentable $\cv$-categories. Recall that an object $A$ of a $\cv$-category $\ck$ is $\lambda$-presentable in the enriched sense if its enriched hom-functor $\ck(A,-):\ck\to\cv$
preserves $\lambda$-filtered colimits. And a $\cv$-category $\ck$ is locally 
$\lambda$-presentable in the enriched sense if it has weighted colimits and a set of $\lambda$-presentable objects in the enriched sense such that every object of $\ck$ is a $\lambda$-filtered colimit of them. Enriched $\cm$-locally $\lambda$-generated categories were introduced in \cite{DR}. Here, $(\ce,\cm)$ should be a $\cv$-factorization system on a $\cv$-category $\ck$  and an object $A$ is $\lambda$-generated in the enriched sense if $\ck(A,-):\ck\to\cv$ preserves $\lambda$-directed colimits of $\cm$-morphisms. Now, a $\cv$-category $\ck$ is $\cm$-locally $\lambda$-generated if it has weighted colimits and a set $\ca$ of $\lambda$-generated objects w.r.t. $\cm$ in the enriched sense such that every object of $\ck$ is a $\lambda$-directed colimit of objects from $\ca$ and $\cm$-morphisms.

We denote by $\Met$ the category of (generalized) metric spaces and nonexpanding mappings. This means that the metric $d$ takes values in $[0,\infty]$ and satisfies the usual axioms. A mapping $f:A\to B$ is nonexpanding if $d(x,y)\geq d(fx,fy)$ for every $x,y\in A$. 
The category $\Met$ is complete and cocomplete with products given by the sup-metric
(max-metric for finite products). Moreover, $\Met$ is locally $\aleph_1$-presentable (see \cite[Examples 4.5(3)]{LR}) and $\aleph_1$-presentable metric spaces are precisely those having a cardinality $\leq\aleph_0$ (see \cite[Proposition 2.6]{AR1}). The only $\aleph_0$-presentable object in $\Met$ is the empty space (see \cite[Remark 2.7(1)]{AR1}).

A metric space is \textit{discrete} if all non-zero distances are $\infty$. Discrete metric
spaces are precisely copowers of the terminal metric space $1$. The metric space $2_\eps$
has two points of distance $\eps$. In particular, $2_\infty=1+1$. Finite metric spaces
$1$ and $2_\eps$ form a dense subcategory of $\Met$, i.e., every metric space is a canonical colimit of copies of $1$ and $2_\eps$.

Following \cite[Remark 2.5(2)]{AR1}), every finite metric space $A$ is finitely generated w.r.t. isometries, which means that $\Met(A,-):\Met\to\Set$ preserves $\lambda$-directed colimits of isometries. In $\Met$, we have (surjective, isometry) factorizations (see 
\cite[Examples 3.16]{AR1}) and, following \cite[Remark 2.5(2)]{AR1}), this factorization system is $\aleph_0$-convenient in the sense of \cite{DR}. Since every metric space is a directed colimit of its finite subspaces, $\Met$ is isometry-locally finitely generated.

The category $\Met$ is symmetric monoidal closed with the tensor product $X\otimes Y$ putting the $+$-metric $(d\otimes d)((x,y),(x',y'))=d(x,x')+d(y,y')$ 
on $X\times Y$. The tensor unit $I$ is the terminal
metric space $1$. The in\-ter\-nal hom equips the hom-set  $\Met(X,Y)$ with the sup-metric $d(f,g)$ given by $\sup\{d(fx,gx)\;|\;x\in X\}$.  
The category $\Met$ is locally $\aleph_1$-presentable as a closed category. Since $\Met(I,-):\Met\to\Set$ creates $\aleph_1$-directed colimits, $\aleph_1$-presentability in the enriched and ordinary sense coincide. 

Let $\PMet$ be the category of (generalized) pseudometric spaces and nonexpansive maps. (The difference is just that for a pseudometric we do not require $d(x,y)>0$ if $x\neq y$.) $\Met$ is a full reflective subcategory of $\PMet$.
 
\begin{lemma}\label{product}
The reflector $F:\PMet\to\Met$ preserves finite products.
\end{lemma}
\begin{proof}
The reflection $FX$ is the metric quotient of $X$ where we identify $x,y\in X$ such that $d(x,y)=0$. Since the distance on a product is given by the max-metric, $d((x,x'),(y,y'))=0$ iff $d(x,x')=d(y,y')=0$. Hence we get the result.
\end{proof}

This was observed in \cite[Corollary 3.10]{AMMU} for 1-bounded metric spaces.

\begin{propo}\label{commute}
In $\Met$, finite products commute with filtered colimits. 
\end{propo}
\begin{proof}
Following \ref{product}, it suffices to prove that finite products commute with filtered colimits in $\PMet$. The proof is the same as that in \cite{AMMU} for 1-bounded metric spaces. In fact, let $X$ be a pseudometric space and $D:\cd\to\PMet$ a filtered diagram. Then it suffices to realize that $X\times -:\PMet\to\PMet$ preserves filtered colimits because the diagonal $D\to D\times D$ is final. 
\end{proof}

\begin{rem}
{
\em
General finite limits do not commute with filtered colimits in $\Met$. For instance,
filtered colimits of monomorphisms are not necessarily monomorphisms. It suffices
to take the identity maps $2_\infty\to 2_{\frac{1}{n}}$ whose colimit is $2_\infty\to 1$.
}
\end{rem}
\begin{coro}\label{finpres} 
Let $A$ be a finite discrete metric space. Then its hom-functor 
$$
\Met(A,-):\Met\to\Met
$$ 
preserves filtered colimits. In particular, $A$ is finitely presentable in the enriched sense.
\end{coro}

\begin{rem}
{
\em
Non-discrete finite metric spaces are not finitely presentable in the enriched sense.
Indeed, given a finite space $(A,d)$ of more than one element, let $d_n$ be the metric on $A$ given by $d_n (x,y)=d(x,y)+ \frac 1n$ for $n=1,2,3,\dots$. This defines an $\omega$-chain
 $$
 (A, d_1) \xrightarrow{\ \id\ } (A, d_2) \xrightarrow{\ \id \ } (A, d_3) \cdots
 $$
 with colimit $(A,d)$. Obviously, the identity morphism of $(A,d)$ does not factorize through any of the  colimit maps $\id \colon (A, d_n) \to (A,d)$. Hence, $(A,d)$ is not finitely presentable
in the enriched sense.
 
Consequently, finite non-discrete cotensors do not commute with filtered colimits in $\Met$. 
}
\end{rem}

\section{Algebras in general categories}
The following definition captures the Lawvere-Linton approach to equational theories.
Given a category $\ck$ and a small dense subcategory $\ca$ of $\ck$ (this includes
that $\ca$ is a full subcategory), the $\ca$-\textit{nerve} functor 
$N_\ca:\ck\to\Set^{\ca^{\op}}$ assigns to every object $K$ the domain restriction of $\ck(-,K)$ on $\ca^{\op}$. Given $X$ in $\ca$ and $K$ in $\ck$, we will denote the set $\ck(X,K)$ as $K^X$. Similarly, $K^f=\ck(f,K)$ and $h^X=\ck(X,h)$.

\begin{defi}[\cite{BG}]\label{pretheory}
{
\em
Let $\ck$ be a category and $\ca$ a small dense subcategory of $\ck$. An $\ca$-\textit{pretheory} is an identity-on-objects functor $J:\ca^{\op}\to\ct$.

A $\ct$-\textit{algebra} is an object $A$ of $\ck$ together with a functor $\hat{A}:\ct\to\Set$ such that $\hat{A}\cdot J=N_\ca(A)$.  

A \textit{homomorphism} of $\ct$-algebras $(A,\hat{A})\to(B,\hat{B})$ is a morphisms $h:A\to B$  
 together with a natural transformation $\hat{h}:\hat{A}\to\hat{B}$ such that 
$\hat{h}J=N_\ca(h)$.   

The resulting category of $\ct$-algebras is denoted as $\Alg(\ct)$. 
}
\end{defi}

\begin{rem}\label{pretheory1}
{
\em
This definition can be unpacked as follows. Objects of $\ca$ are \textit{arities}.
Morphisms $\omega:J(X)\to J(Y)$ of $\ct$ are $(X,Y)$-ary operations where $X$ is
the \textit{input arity} and $Y$ is the \textit{output arity} of $\omega$. 
In particular, a morphism $f:Y\to X$ gives the $(X,Y)$-ary operation $J(f)$. 
A $\ct$-algebra is an object $A$ of $\ck$ together with mappings $\omega_A:A^X\to A^Y$
for every $\omega:X\to Y$ in $\ct$ such that
\begin{enumerate}
\item $J(f)_A= A^f$, and
\item $(\tau\omega)_A=\tau_A\omega_A$.
\end{enumerate}
Homomorphism $h:A\to B$ of $\ct$-algebras are morphisms $h:A\to B$ such that $h^Y\omega_A=\omega_B h^X$ for every morphism $\omega:X\to Y$ of $\ct$.

Let $\ck=\Set$ and $\ca$ consist of finite sets. Then a $\ca$-pretheories coincide
with Lawvere theories. Morphism $J(n)\to J(m)$ are $m$-tuples of $n$-ary operations.
In the language of universal algebra, Lawvere theories correspond to clones of finitary
operations where compositions are finitary multiple compositions, i.e., superpositions
of operations. Mappings $1\to m$ provide variables $x_1,\dots,x_m$. 

In the same way as clones are presented by equational theories we can present pretheories
by equational theories over $\ck$.
}
\end{rem}

\begin{defi}[\cite{R}]\label{theory}
{
\em
Let $\ck$ be a category and $\ca$ a full subcategory of $\ck$. An $\ca$-\textit{ary type} $t$ over $\ck$ is a class $\Omega$ equipped with a mapping $t:\Omega\to\ob(\ca)\times\ob(\ca)$ such that $\Omega^{X,Y}=t^{-1}(X,Y)$ are sets for every $X,Y\in\ob(\ca)$. Elements of $\Omega^{X,Y}$ are called $(X,Y)$-\textit{ary operation symbols} of type $t$.

\textit{Terms} of type $t$ are inductively defined as follows:
\begin{enumerate}
\item Every $\omega\in\Omega$ is a $t(\omega)$-ary term,
\item Every morphism $f:Y\to X$ of $\ca$ determines an $(X,Y)$-ary term $x_f$, 
\item If $p$ is an $(X,Y)$-ary term and $q$ an $(Y,Z)$-ary term then $qp$ is an $(X,Z)$-ary term, 
\item $x_{fg}=x_gx_f$, 
\item $(pq)r=p(qr)$, and
\item If $p$ is an $(X,Y)$-ary term then $x_{\id_Y}p=p=px_{\id_X}$.
\end{enumerate}

\textit{Equations} $p=q$ of type $t$ are pairs $(p,q)$ of terms of type $t$. An \textit{equational theory} of type $t$ is a class $E$ of equations of type $t$. 
}
\end{defi}

\begin{defi}
{
\em
An \textit{algebra} of type $t$ is an object $A$ of $\ck$ equipped with mappings 
$\omega_A:A^X\to A^Y$ for every $(X,Y)$-ary operation symbol $\omega$ of $t$.

Terms are interpreted in $A$ as follows:
\begin{enumerate}
\item $(x_f)_A= A^f$, and
\item $(qp)_A=q_Ap_A$.
\end{enumerate}

An algebra $A$ \textit{satisfies} an equation $p=q$ if $p_A=q_A$. It satisfies a theory $E$ if it satisfies all equations of $E$.

A \textit{homomorphism} $h:A\to B$ of $t$-algebras is a morphism $h:A\to B$ such that $h^Y\omega_A=\omega_B h^X$ for every $\omega\in\Omega^{X,Y}$.

$\Alg(E)$ will denote the category of $E$-algebras and $U_E:\Alg(E)\to\ck$ will be the forgetful functor.
}
\end{defi}

\begin{rem}
{
\em
The definition of an equational theory in \cite{R} does not contain conditions (5) an (6). But it does not influence the concept of an algebra because the composition of mappings is associative and $x_{\id}$ are interpreted as identities. 

Let $\ct$ be an $\ca$-pretheory and take $\Omega$ as the set of morphisms of $\ct$ and 
$t(\omega)=(X,Y)$ for $\omega:J(X)\to J(Y)$. The resulting $\ca$-ary type $t$ has the same algebras as $\ct$ and its terms coincide with operation symbols.

Conversely, consider an $\ca$-ary type $t$. Let $\tilde{\Omega}$ denote the category of terms of type $t$ where $\tilde{\Omega}(X,Y)$ consists of $(X,Y)$-ary terms. Compositions are given by (3) and identities are $x_{\id}$. Our definition makes $\tilde{\Omega}$ a category and $x_-:\ca^{\op}\to\tilde{\Omega}$ is a functor which is an identity on objects. Then an algebra $A$ is equipped with a functor $\tilde{\Omega}\to\Set$ whose composition with $x_-$ is $\ck(-,A)$. Hence $t$-algebras coincide with $\tilde{\Omega}$-algebras.

Given an equational theory $E$ of type $t$, let $\tilde{\Omega}_E=\tilde{\Omega}/E$
be the quotient of $\tilde{\Omega}$ given by $E$. Again, $E$-algebras coincide with
$\tilde{\Omega}/E$-algebras.
Thus our equational theories and algebras coincide with pretheories and algebras in the sense of \cite{BG}.
}
\end{rem}

\begin{defi}
{
\em
Let $\ck$ be a locally $\lambda$-presentable category and $\ck_\lambda$ denote the (representative) small full subcategory consisting of $\lambda$-presentable objects.
Then $\ck_\lambda$-ary types will be called $\lambda$-\textit{ary}. Similarly, equational
theories of $\lambda$-ary types are called $\lambda$-ary.
}
\end{defi}

\begin{exams}\label{first}
{
\em
(1) Let $\Pos$ be the category of posets and monotone mappings. $\Pos$ is a locally finitely presentable category. Consider a type $t$ over $\Pos$ and $\omega\in\Omega^{X,Y}$ be its $(X,Y)$-ary operation. Elements of $Y$ correspond to monotone mappings $f:1\to Y$. Then terms $x_f\omega$ are usual $X$-ary operations. For a $t$-algebra $A$, $A^X$ is the set of monotone mapping $a:X\to A$. Then $\omega_A:A^X\to A^Y$ is a mapping. We can replace $\omega$ by an $Y$-tuple of $X$-ary terms $x_f\omega$, $f:1\to Y$. But we have to force that $\omega_A(a):Y\to A$ is monotone for every monotone $a:X\to A$. This means that $x_f\omega(a)\leq x_g\omega(a)$ for every $f\leq g$, $f,g:1\to Y$ and every $a\in A^X$. But this is the same as $x_f\omega\leq x_g\omega$ for every $f\leq g$, $f,g:1\to Y$. Of course, $f\leq g$ iff the corresponding elements $y,z\in Y$ satisfy $y\leq z$.

Consequently, we can reduce $(X,Y)$-ary operations to $X$-ary operations but we have to replace equations in equational theories by inequations $p\leq q$ of terms. Conversely, every such inequational theory can be replaced by an equational theory with $(X,Y)$-ary operations. In fact, given $X$-ary terms $p$ and $q$, then an inequation $p\leq q$ corresponds to the existence of a $(X,2)$-ary operation $r$ such that $x_0r=p$ and $x_1 r=q$ where $2$ is a two-element chain $\{0,1\}$. Here $0:1\to 2$ has the value $0$ and
$1:1\to 2$ has the value $1$.

Hence our finitary equational equational theories over $\Pos$ correspond to usual
inequational theories for ordered algebras (see \cite{AFMS}).

(2) Consider a type $t$ over $\Met$. Analogously to (1), we can reduce $(X,Y)$-ary operations to $X$-ary ones but we have to ensure that $\omega_A(a)$ is nonexpanding for every nonexpanding $a:X\to A$. This means that
$$
d(x_f\omega(a),x_g\omega(a))\leq d(f,g)
$$
for every $f,g:1\to Y$ and every $a\in A^X$. This is the same as 
$$
d(x_f\omega,x_g\omega)\leq d(f,g).
$$
Of course, $d(f,g)$ is the distance $d(y,z)$ of the corresponding elements.

Consequently, we can reduce $(X,Y)$-ary operations to $X$-ary operations but we have to replace equations in equational theories by "quantitative equations" $p=_\eps q$ of terms. Conversely, every such metric equational theory can be replaced by an equational theory with $(X,Y)$-ary operations. In fact, given $X$-ary terms $p$ and $q$, then a quantitative equation $p=_\eps q$ corresponds to the existence of a $(X,2_\eps)$-ary operation $r$ such that $x_0r=p$ and $x_1 r=q$. 

Hence our equational theories over $\Met$ correspond to "metric equational" theories for metric algebras. The special case when the input arities $X$ are finite discrete metric spaces and the output arities $Y$ are finite metric spaces was considered in \cite{W}, or \cite{H} where the resulting categories $\Alg(E)$ were called varieties of metric algebras. If the arities $Y$ are also discrete, usual equations are used.  
}
\end{exams}

\begin{theo}\label{monad0}
Let $\ck$ be a locally $\lambda$-presentable and $E$ be a $\lambda$-ary equational theory over $\ck$. Then $\Alg(E)$ is locally $\lambda$-presentable and $U_E$ is monadic and preserves $\lambda$-filtered colimits.
\end{theo}
\begin{proof}
Let $\Omega$ consist of a single $(X,Y)$-ary operation symbol where $X$ and $Y$ are $\lambda$-presentable. Then $\Alg(\Omega)$ is the inserter $\Ins(-^X,-^Y)$ which is an accessible category (see \cite[Theorem 2.72]{AR}). Clearly, $\Alg(\Omega)$ is complete and $U_\Omega$ preserves limits and $\lambda$-directed colimits. Thus $U_\Omega$ has a left adjoint $F$ (see \cite[Theorem 1.66]{AR}). Since $U$ is conservative and faithful, $FZ$, $Z$ $\lambda$-presentable, form a strong generator of $\Alg(\Omega)$ consisting of $\lambda$-presentable objects. Following \cite[Theorem 1.20 and Corollary 2.47]{AR}, $\Alg(\Omega)$ is locally $\lambda$-presentable. Since $U_\Omega$ creates coequalizers of $U_\Omega$-absolute pairs, $U_\Omega$ is monadic.

For a general $\lambda$-ary type, $\Alg(\Omega)$ is a multiple pullback of $U_{\{f\}}$, $f\in\Omega$. Since $U_{\{f\}}$ are transportable, this multiple pullback is a multiple
pseudopullback (see \cite[Proposition 5.1.1]{MP}). Hence $\Alg(\Omega)$ is locally $\lambda$-presentable (see \cite[Theorem 2.72]{AR}). Again, $U_\Omega$ preserves limits, $\lambda$-filtered colimits and creates coequalizers of $U_\Omega$-absolute pairs. Hence it is monadic.

Finally, $\Alg(E)$ forms a full reflective subcategory of $\Alg(\Omega)$ closed under $\lambda$-filtered colimits.  Hence it is locally $\lambda$-presentable (see 
\cite[Corollary 2.48]{AR}). The rest is obvious.
\end{proof}

\begin{rem}
{
\em
We could shorten the proof by using an unpublished \cite{B} where 2.14 shows that the inserters $\Ins(-^X,-^Y)$ are locally $\lambda$-presentable.
}
\end{rem}

\begin{coro}\label{monad1}
Let $\ck$ be a locally $\lambda$-presentable and $E$ be a $\lambda$-ary equational theory over $\ck$. Then the corresponding monad $T_E$ is $\lambda$-ary.
\end{coro}

The following observation is crucial.

\begin{obs}\label{compare}
{
\em
Let $U:\ca\to\ck$ be a faithful functor, i.e., $\ca$ is \textit{concrete} over $\ck$. For $X,Y\in\ob(\ck)$, let $\Omega^{X,Y}$ consist of natural transformations $U^X\to U^Y$ where $U^X=\ck(X,U-)$. Since the collections $\Omega^{X,Y}$ do not need to be sets, the resulting "type" is not legitimate. Ignoring this, we get terms where $x_f=U^f$ and $\omega\omega'$ is the composition. Observe that $\tilde{\Omega}(X,Y)=\Omega^{X,Y}$.
This yields the "equational theory" $E_U$ and the functor $H_U:\ca\to\Alg(E_U)$ such that $H_U(A)$ is $A$ equipped with components $\omega_A$ as operations.

Assume that $U$ has a left adjoint $F$. Then the natural transformations $\omega:U^X\to U^Y$ correspond to natural transformations $\ca(FX,-)\to\ca(FY,-)$ and thus to morphisms $FY\to FX$. Hence we get a type $t_U$ and an equational theory $E_U$. 

This is due to Linton \cite{L} and it goes back to Lawvere \cite{La}. The fundamental result
of Linton \cite{L} is that if $U$ is monadic then the comparison functor $H_U:\ck^T\to\Alg(E_U)$ from the category of $T$-algebras is an equivalence. Moreover, given a $T$-algebra, elements $a\in A^X$ correspond to morphisms $FX\to A$.

Observe that Linton's result exactly says that the full subcategory consisting of free algebras is dense in the category of algebras. This can be proved as follows. Consider the canonical presentation  
$$ 
	\xymatrix@=4pc{
		&  
		FUFUA\ar@<0.5ex>[r]^{\varepsilon_{FUA}}
		\ar@<-0.5ex>[r]_{FU\varepsilon_A}& FUA  \ar[r]^{\varepsilon_A} & A
	}
	$$
of a $T$-algebra $A$. Let $\varphi_a:FX\to B$ be a cocone from the canonical diagram of free $T$-algebras to $A$; here $a:FX\to A$. Let $\tilde{a}:X\to UA$ be the transpose 
of $a$. There is a unique $t:A\to B$ such that $t\varepsilon_A=\varphi_{\varepsilon_A}$. Since
$$
ta=t\varepsilon_AF\tilde{a}=\varphi_{\varepsilon_A}F\tilde{a}=\varphi_a,
$$
$A$ is a colimit of its canonical diagram.
}
\end{obs}

\begin{theo}\label{monad2}
Let $\ck$ be a locally $\lambda$-presentable category and $T$ be a $\lambda$-ary monad on $\ck$. Then there is a $\lambda$-ary equational theory $E$ such that the concrete categories (over $\ck$) $\ck^T$ and $\Alg(E)$ are equivalent.
\end{theo}
\begin{proof}
Following \ref{compare}, $\ck^T$ is equivalent to $\Alg(E_U)$ where $U:\ck^T\to\ck$ is the forgetful functor. Moreover $(X,Y)$-ary operations of type $t_U$ correspond to morphisms 
$FY\to FX$. Let $E$ be the subtheory of $E_U$ where operations are $\lambda$-ary, i.e., $X$ and $Y$ are $\lambda$-presentable. We get the reduct functor $R:\Alg(E_U)\to\Alg(E)$
forgetting operations which are not $\lambda$-ary.

Given $X,Y\in\ob(\ck)$, we express them as $\lambda$-directed colimits $(x_i:X_i\to X)_{i\in I}$ and $(y_j:Y_j\to Y)_{j\in J}$ of $\lambda$-presentable objects in $\ck$. This yields $\lambda$-directed colimits $(Fx_i:FX_i\to FX)_{i\in I}$ and $(Fy_j:FY_j\to FY)_{j\in J}$ in $\ck^T$. Since $U$ preserves $\lambda$-directed colimits,
$F$ preserves $\lambda$-presentable objects. Consequently, every $f:FY\to FX$ is a $\lambda$-directed colimit of $f_{ji}:FY_{j} \to FX_{i_j}$. Indeed, given $j\in J$ and $i\in I$, since $FY_j$ is $\lambda$-presentable, there is $i_j\geq i$
such that $fFy_j=F(x_{i_j})f_{ij}$. This implies that $R$ is an equivalence.

Indeed, let $\omega$ be an $(X,Y)$-ary operation corresponding to $f:FY\to FX$
and $\omega_{ij}$ be $(X_{i_j},Y_j)$-ary operations corresponding to 
$f_{ij}:FY_j\to FX_{i_j}$. We have $A^Y=\lim_j A^{Y_j}$. Since
$$
(\omega_{i j})_A A^{x_{i_j}}: A^X\to A^{Y_j}
$$
form a cone, there is a unique $\omega_A:A^X\to A^Y$ such that 
$$
A^{y_j}\omega_A=(\omega_{i j})_A A^{x_{i_j}}.
$$
This clearly makes $A$ an $E_U$-algebra.
\end{proof}

\begin{rem}
{
\em
(1) We have got a one-to-one correspondence between $\lambda$-monads on $\ck$ and $\lambda$-ary equational theories over $\ck$. As explained in the introduction, this result follows from \cite{BG}, Corollary 21 and Proposition 23.

(2) \ref{monad2} and \ref{monad0} imply that $\ck^T$ is locally $\lambda$-presentable, which is already in \cite{GU}.
}
\end{rem}

\begin{defi}\label{preservation0}
{
\em
We say that a type $t$ over $\ck$ is $(\lambda',\lambda)$-\textit{ary} if all arities $X$ are $\lambda'$-presentable and all arities $Y$ of its operation symbols are $\lambda$-presentable. 
}
\end{defi}

\begin{propo}\label{preservation}
Let $\ck$ be a locally $\lambda$-presentable category, $\lambda'\leq\lambda$ be a regular cardinal and $E$ be a $(\lambda',\lambda)$-ary equational theory over $\ck$. Then $U_E$ preserves $\lambda'$-directed colimits. Hence the corresponding monad is $\lambda'$-ary.
\end{propo}
\begin{proof}
Consider a $\lambda$-directed diagram $(a_{ij}:A_i\to A_j)_{i\leq j\in I}$ 
in $\Alg(E)$. Take $U_EA=\colim U_EA_i$ with a colimit cocone $a_i:U_EA_i\to U_EA$ and an $(X,Y)$-ary $t$-operation $\omega$. We have a bijection
$u:\colim(U_EA_i)^X\to(\colim U_EA_i)^X$ and a mapping $v:\colim(U_EA_i)^Y\to(\colim U_EA_i)^Y$. We define $\omega_A:A^X\to A^Y$ as the composition $v\cdot\colim\omega_{A_i}\cdot u^{-1}$.
It is easy to see that we get a colimit $(\bar{a}_i:A_i\to A)$ in $\Alg(E)$ with $U\bar{a}_i=a_i$ for every $i\in I$.
\end{proof}

\begin{exams}\label{ex}
{
\em
(1) Over $\Set$, $(X,Y)$-ary operations coincide with $Y$-tuples of usual $X$-ary operations, i.e. $(X,1)$-ary in our terminology. Thus we get the well-known correspondence between Lawvere theories and finitary monads in $\Set$ due to \cite{L1}.

(2)  Since $\Pos$ is locally finitely presentable, \ref{monad1} and \ref{monad2} imply that finitary monads on $\Pos$ correspond to finitary inequational theories, which was recently proved in \cite{AFMS}. 

(3) Since $\Met$ is locally $\aleph_1$-presentable, monads preserving $\aleph_1$-directed colimits on $\Met$ correspond to $\aleph_1$-ary equational theories.  
}
\end{exams}

\begin{rem}\label{reduction}
{
\em
Examples \ref{ex} can be extended. Let $E$ be a $\lambda$-ary equational theory over a locally $\lambda$-presentable category $\ck$ and $\ca\subseteq\ck_\lambda$ be a dense subcategory consisting of (some) $\lambda$-presentable objects. Let $E'$ be a subtheory of $E$ consisting of $(X,Y)$-ary operations where $Y\in\ca$. Then the reduct functor $\Alg(E)\to\Alg(E')$ is an equivalence.

It suffices to express $Y$ as a colimit $\delta_d:Dd\to Y$ of its canonical diagram $D:\cd\to\ca$.
Then, for an $E'$-algebra $A$ and an $(X,Y)$-ary operation $\omega$ of $E$, 
$$
A^Y=A^{\colim D}\cong\lim A^D
$$ 
with projections $A^{\delta_d}:A^Y\to A^{Dd}$. This yields a unique $\omega_A:A^X\to A^Y$ making $A$ a $E$-algebra.

In \ref{ex}(1), one takes $\ca=\{1\}$, in \ref{ex}(2), $\ca=\{1,2\}$ and, 
in \ref{ex}(3), $\ca=\{1,2_\eps\}$ with $\eps>0$ .
}
\end{rem}

\begin{defi}
{
\em
Let $\ck$ be a $\cm$-locally $\lambda$-generated category in the sense of \cite{DR}. We say that a type $t$ over $\ck$ is $\lambda$-\textit{ary} if all arities $X$ and $Y$ of its operation symbols are $\lambda$-generated w.r.t. $\cm$. A theory is $\lambda$-ary if its type is $\lambda$-ary. 
}
\end{defi}

\begin{theo}\label{monad3}
Let $\ck$ be $\cm$-locally $\lambda$-generated and $E$ be a $\lambda$-ary equational theory over $\ck$. Then $\Alg(E)$ is locally $\lambda$-generated and $U_E$ is both monadic and a morphism of locally $\lambda$-generated categories.
\end{theo}
\begin{proof}
Following \cite{DR}, $\ck$ is locally $\mu$-presentable for some $\mu\geq\lambda$. The theory $E$ is then $\mu$-ary. Hence $\Alg(E)$ is locally $\mu$-presentable and $U_E$ is monadic (see \ref{monad0}). 

Since $\ck$ is $\cm$-locally $\lambda$-generated, $\ck$ is equipped with a $\lambda$-convenient factorization system $(\ce,\cm)$. Put $\cm'=U_E^{-1}(\cm)$. Like in the proof of \cite[Theorem 2.23]{DR}, $(\colim F(\ce),\cm')$ is a $\lambda$-convenient factorization system in $\Alg(T)$; here $F$ is left adjoint to $U_E$. Since $U_E$ is faithful and conservative, $F$ maps a strong generator $\cg$ of $\ck$ consisting of $\lambda$-generated objects w.r.t. $\cm$ to a strong generator $F(\cg)$ in $\Alg(T)$. Since $U_E$ sends $\lambda$-directed colimits of $\cm'$-morphisms to $\lambda$-directed colimits of $\cm$-morphisms, objects of $F(\cg)$ are $\lambda$-generated w.r.t. $\cm'$ (see the proof of 
\cite[Lemma 3.11]{DR}). Following \cite[Theorem 2.22]{DR}, $\Alg(T)$ is $\cm'$-locally $\lambda$-generated. Clearly, $U_E$ is a morphisms of locally $\lambda$-generated categories. 
\end{proof}

\begin{rem}
{
\em
The resulting monad $T_E=U_EF$ is $\mu$-ary. But we do not know that it preserves $\lambda$-directed colimits of $\cm$-morphisms. This means that it sends $\lambda$-directed colimits of $\cm$-morphisms to $\lambda$-directed colimits. For this, we would need that $F$ sends $\cm$-morphisms to $\cm'$-morphisms. Then $T_E$ would also preserve $\cm$-morphisms. 
}
\end{rem}

\begin{theo}\label{monad4} 
Let $\ck$ be $\cm$-locally $\lambda$-generated and locally $\mu$-presentable for $\lambda\leq\mu$. Let $T$ be a $\mu$-ary monad on $\ck$ preserving $\cm$-morphisms and $\lambda$-directed colimits of $\cm$-morphisms. Then there is a $\lambda$-ary equational theory $E$
such that $\ck^T$ and $\Alg(E)$ are equivalent (as concrete categories over $\ck$).
\end{theo}
\begin{proof}
The proof is analogous to that of \ref{monad2}. Only, given $X,Y\in\ob(\ck)$, we express them as
$\lambda$-directed colimits $(x_i:X_i\to X)_{i\in I}$ and $(y_j:Y_j\to Y)_{j\in J}$ of $\lambda$-generated objects w.r.t. $\cm$ and $\cm$-morphisms in $\ck$. Then $(Fx_i:FX_i\to FX)_{i\in I}$
and $(Fy_j:FY_j\to FY)_{j\in j}$ are $\lambda$-directed colimits of $\cm'$-morphisms. Since $FY_j$ are $\lambda$-generated w.r.t. $\cm'$, $f:FY\to FX$ is a $\lambda$-directed colimit of $f_i:FY_{j_i}\to X_i$.
\end{proof}

\begin{rem}
{
\em
Since $\Met$ is isometry-locally finitely generated, every $\aleph_1$-ary monad $T:\Met\to\Met$ preserving isometries and $\aleph_0$-directed colimits of isometries is given by an $\aleph_0$-ary equational theory over $\Met$. This means that arities are finite metric spaces. Conversely, every $\aleph_0$-ary equational theory yields an $\aleph_1$-ary monad $T:\Met\to\Met$. But we do not know whether it also preserves isometries and $\aleph_0$-directed colimits of isometries. The next example shows that
it does not need to happen.
}
\end{rem}

\begin{exam}\label{counter}
{
\em
Let $X$ be the metric space having points $a,b,c$ with distances $(a,b)=d(b,c)=1$ and $d(a,c)=2$
and $Y=2_1$ be the metric space with two points $0,1$ having the distance $1$. Let $f,g:Y\to X$ be given as follows: $f(0)=a, f(1)=b=g(0)$ and $g(1)=c$. Let $E$ be given by $x_f=x_g$. This equation
corresponds to the formula
$$
(\forall x,y,z) (d(x,y)\leq 1\wedge d(x,z)\leq 1)\Rightarrow y=z),
$$
which is equivalent to
$$
(\forall x,y)(d(x,y)\leq 1\Rightarrow x=y).
$$
Hence $E$-algebras are metric spaces having distances of distinct points $>1$ and the monad $T_E$ is given by the reflection to this reflective subcategory. Consider the isometry $m:2_2\to X$, where $2_2$ has two points having the distance $2$. Then $T_E(m)$ is the constant mapping. Hence
$T_E$ does not preserve isometries although it is given by a finitary equational theory. Moreover,
$T_E$ preserves directed colimits of isometries but it does not preserve directed colimits.
}
\end{exam}

\section{Enriched equational theories}
Let $\cv$ be a symmetric monoidal closed category and $\ck$ a $\cv$-category. Given  objects $A$ and $X$ in $\ck$, we denote the $\cv$-object $\ck(X,A)$ as $A^X$. 
Similarly, $K^f=\ck(f,K)$ and $h^X=\ck(X,h)$ are morphisms in $\cv$. Let $\ck_0$ be the underlying category of $\ck$.  

At first, we will define enriched algebras over unenriched equational theories from \ref{theory}.

\begin{defi}\label{ealg}
{
\em
Let $t$ be a type over $\ck_0$. A \textit{$\cv$-algebra} of type $t$ is an object $A$ of $\ck_0$ equipped with morphisms $\omega_A:A^X\to A^Y$ in $\cv_0$ for every $(X,Y)$-ary operation symbol of $t$. Terms are interpreted in $A$ as in \ref{theory}, just the meaning of $A^f$ is changed. 

A $\cv$-algebra $A$ \textit{satisfies} an equation $p=q$ if $p_A=q_A$. It satisfies a theory $E$ if it satisfies all equations of $E$.

A \textit{homomorphism} $h:A\to B$ of $\cv$-algebras of type $t$ are morphisms $h:A\to B$ such that $h^Y\omega_A=\omega_B h^X$ for every $\omega\in\Omega^{X,Y}$.

$\Alg(E)$ will denote the category of $\cv$-algebras of $E$ and $U_E:\Alg(E)\to\ck_0$ will be the forgetful functor.
}
\end{defi}

\begin{rem}
{
\em
A $\cv$-algebra $A$ is equipped with a functor $\tilde{\Omega}\to\cv_0$ whose composition with 
$x_-$ is $\ck_0(-,A)$.
}
\end{rem}

Enriched equational theories are defined as follows (they correspond 
to $\cv$-pretheories of \cite{BG}).

\begin{defi}[\cite{BG}]
{
\em
A $\cv$-\textit{equational theory} over $\ck$ is a $\cv$-category $\Theta$ together with 
an identity-on-objects  $\cv$-functor $x_-:\ca^{\op}\to\Theta$ where $\ca$ is a full subcategory of $\ck$. A $\Theta$-algebra is an object $A$ of $\ck$ together with 
a $\cv$-functor $\hat{A}:\Theta\to\cv$ whose composition with $x_-$ is $\ck(-,A)$. A homomorphism of $\Theta$-algebras $(A,\hat{A})\to (B,\hat{B})$ is a morphism $h:A\to B$  
 together with a $\cv$-natural transformation $\hat{h}:\hat{A}\to\hat{B}$ such that
$\hat{h}J$ is $\ck(-,h)$ restricted on $\ca^{\op}$. $\Alg(\Theta)$ will denote the $\cv$-category of $\Theta$-algebras and $U_\Theta:\Alg(\Theta)\to\cv$ will be the forgetful $\cv$-functor. 
}
\end{defi}

\begin{rem}
{
\em
A $\cv$-equational theory $\Theta$ over $\ck$ determines an equational theory $\Theta_0$ over $\ck_0$ and every $\Theta$-algebra is a $\cv$-algebra of $\Theta_0$.
 
We will show that, over $\Pos$ or $\Met$, enriched equational theories can be reduced
to unenriched ones having the same enriched algebras.
}
\end{rem}

$\cv$ is locally $\lambda$-presentable as a closed category if it is locally
$\lambda$-presentable, the tensor unit $I$ is $\lambda$-presentable and the tensor product of two $\lambda$-presentable objects is also $\lambda$-presentable. An object $X$  
of a $\cv$-category  $\ck$ is $\lambda$-presentable (in the enriched sense) if $\ck(X,-):\ck\to\cv$ preserves $\lambda$-directed colimits. Since the set-valued hom-functor $\cv(X,-)$ is the composition of the $\cv$-hom-functor $\cv(X,-)$ and the set-valued hom-functor $\cv(I,-)$, $\lambda$-presentability in the enriched sense implies $\lambda$-presentability in the ordinary sense. If $\cv(I,-):\cv\to\Set$ creates $\lambda$-directed colimits, the both concepts are equivalent.

\begin{defi}
{
\em 
We say that a type $t$ over $\ck$ is $\lambda$-\textit{ary} if all arities $X$ and $Y$ of its operation symbols are $\lambda$-presentable in $\cv$. A theory is $\lambda$-ary if its type is $\lambda$-ary. Similarly, a $\lambda$-ary $\cv$-equational theory is a 
$\cv$-category $\Theta$ together with an identity-on-objects $\cv$-functor $x_-:\ck_\lambda^{\op}\to\Theta$. 
}
\end{defi}

\begin{theo}[\cite{BG}]\label{emonad0}
Let $\cv$ be a locally $\lambda$-presentable as a closed category, $\ck$ a locally $\lambda$-presentable $\cv$-category and $\Theta$ be a $\lambda$-ary $\cv$-equational theory over $\ck$. Then $\Alg(\Theta)$ is a locally $\lambda$-presentable $\cv$-category and $U_E$ is monadic and preserves $\lambda$-directed colimits. Hence the corresponding $\cv$-monad $T_E$ on $\ck$ is $\lambda$-ary.
\end{theo}
 
\begin{obs}\label{ecompare}
{
\em
Let $U:\ca\to\cv$ be a faithful $\cv$-functor, i.e., $\ca$ is \textit{$\cv$-concrete}. For
$X,Y\in\ob(\cv)$, let $\Omega^{X,Y}$ consist of natural transformations $U^X\to U^Y$ where
$U^X=\cv(X,U-)$. Like in \ref{compare}, we get a $\cv$-equational theory $E_U$ and the $\cv$-functor $H_U:\ca\to\Alg(E_U)$ such that $H_U(A)$ is $A$ equipped with components $\omega_A$ as operations.

Assume that $U$ has a left $\cv$-adjoint $F$. Then $\cv$-natural transformations $\omega:U^X\to U^Y$ correspond to $\cv$-natural transformations $\ca(FX,-)\to\ca(FY,-)$ and thus to morphisms $FY\to FX$. Hence we get a type $t_U$ and a $\cv$-equational theory $E_U$. Again,
the comparison $\cv$-functor $H_U:\ck^T\to\Alg(T_U)$ is an equivalence of $\cv$-categories. Moreover, given a $T$-algebra, elements $a\in A^X$ correspond to morphisms $FX\to A$. 

This exactly says that the full subcategory consisting of free algebras is $\cv$-dense in the category of algebras, which follows from \cite[Theorem 5.1]{K} because $\ck^T$ is cotensored. 
}
\end{obs}

\begin{theo}[\cite{BG}]\label{emonad2}
Let $\cv$ be a locally $\lambda$-presentable as a closed category, $\ck$ be a locally $\lambda$-presentable $\cv$-category and $T$ be a $\lambda$-ary $\cv$-monad on $\ck$. Then there is a $\lambda$-ary $\cv$-equational theory $E$ over $\cv$ such that the $\cv$-concrete categories $\ck^T$ and $\Alg(E)$ are equivalent.
\end{theo}
Using \ref{ecompare}, the proof is analogous to that of \ref{monad2}.  

\begin{rem}
{
\em
Let $\cv$ be locally $\lambda$-presentable as a closed category. Then we have a one-to-one correspondence between $\lambda$-ary $\cv$-monads on $\cv$ and $\lambda$-ary $\cv$-equational theories over $\cv$. This was proved for $\lambda=\aleph_0$ in \cite{P} and the passage from $\aleph_0$ to $\lambda$ does not create any difficulties. This also follows from the more general result 7.7 in \cite{LaR} and, of course, from \cite{BG}, Corollary 21 and Proposition 23.
}
\end{rem}

\begin{rem}\label{epreservation}
{
\em
Like in \ref{preservation}, if $E$ is $(\lambda,\lambda')$-ary, $\lambda\leq\lambda'$, then
$\Alg(E)$ is locally $\lambda'$-presentable and $U_E$ preserves $\lambda$-directed colimits.
}
\end{rem}

\begin{exams}
{
\em
(1) The category $\Pos$ is locally finitely presentable as a (cartesian) closed category. Hence finitary enriched monads on $\Pos$ correspond to finitary enriched equational theories over $\Pos$. Since $\cv(I,-):\Pos\to\Set$ creates directed colimits, finitely presentable objects in the enriched sense and in the ordinary sense coincide and are equal to finite posets. Given a type $t$ over $\Pos$ and an enriched  $t$-algebra $A$
in the sense of \ref{ealg}, then $A^X$ is the set of monotone mapping $a:X\to A$. But $\omega_A:A^X\to A^Y$ is a monotone mapping for every term $\omega$. Like in 
\ref{first}(1), we can reduce $(X,Y)$-ary operations to $X$-ary operations but we have to replace equations by inequation. In the terminology of \cite{AFMS}, the resulting algebras are \textit{coherent}, i.e., operations $A^X\to A$ are monotone mappings. In this way, an equational theory over $\Pos$ can be extended to an enriched equational
theory over $\Pos$ having the same enriched algebras. 

As in \ref{first}(1) an inequation $p\leq q$ of $X$-ary terms is expressed by factorizing  the $(X,2_0)$-ary term $(p,q)$ through a $(X,2)$-ary operation $r$ as $(p,q)=x_ur$
where $u:2_0\to 2$ is the identity map from the two-element antichain $2_0$ to 
the two-element chain $2$.
Hence we get that finitary enriched monads on $\Pos$ correspond to finitary inequational theories for coherent algebras, which was proved in \cite{AFMS}.
 
In $\Pos$, finite products commute with reflexive coinserters (see \cite{Bou}). Every finite poset is a reflexive coinserter of finite discrete posets and hom-functors $\Pos(A,-):\Pos\to\Pos$ of finite discrete posets $A$
preserve reflexive coinserters. An argument analogous to \ref{preservation} implies that the $\Pos$-monad $T_E$ given by an enriched finitary equational theory over $\Pos$ with $(X,Y)$-ary operations where $X$ is discrete preserves not only filtered colimits but also reflexive coinserters. \cite{ADV} showed that, conversely, every enriched monad $T$ on $\Pos$ preserving filtered colimits and reflexive coinserters is given by such an enriched equational theory. In our setting, this can be proved as follows.

Following \ref{emonad2}, $\ck^T$ is equivalent to $\Alg(E_U)$ where $E_U$ is a finitary
enriched equational theory whose $(X,Y)$-ary operations correspond to morphisms 
$f:FY\to FX$; here $U:\ck^T\to\Pos$ is the forgetful functor and $F$ is its enriched left adjoint. Express $Y$ and $X$ as reflexive coinserters of finite discrete posets
$$ 
	\xymatrix@=4pc{
		&  
	Y_1	\ar@<0.5ex>[r]^{p_0}
		\ar@<-0.5ex>[r]_{p_1}& Y_0  \ar[r]^{p} & Y
	}
	$$
and
$$ 
	\xymatrix@=4pc{
		&  
	X_1	\ar@<0.5ex>[r]^{q_0}
		\ar@<-0.5ex>[r]_{q_1}& X_0  \ar[r]^{q} & X
	}
	$$ 
where $Y_0,X_0$ give elements and $Y_1,X_1$ give $\leq$ (see \cite[Remark 2.3]{ADV}).	
Consider $f:FY\to FX$. Since $U$ preserves reflexive coinserters, $\ck^T(FY_0,-)$
preserves reflexive coinserters. Since
$$ 
	\xymatrix@=4pc{
		&  
	FX_1	\ar@<0.5ex>[r]^{Fq_0}
		\ar@<-0.5ex>[r]_{Fq_1}& FX_0  \ar[r]^{Fq} &F X
	}
	$$ 
is a reflexive coinserter,
$$ 
	\xymatrix@=4pc{
		&  
	\ck^T(FY_0,FX_1)	\ar@<0.5ex>[r]^{\ck^T(FY_0,Fq_0)}
		\ar@<-0.5ex>[r]_{\ck^T(FY_0,Fq_1)}& \ck^T(FY_0,FX_0)  \ar[r]^{\ck^T(FY_0,Fq)} &\ck^T(FY_0,FX)
	}
	$$ 
is a reflexive coinserter. Hence $\ck^T(FY_0,Fq)$ is surjective. Thus there exists
$$
f_0:FY_0\to FX_0
$$ 
such that $F(q)f_0=fF(p)$. Since $p_0\leq p_1$, we have $Fp_0\leq Fp_1$ and thus 
$f_0F(p_0)\leq f_0F(p_1)$.

Like in the proof of \ref{monad2}, let $E$ be the subtheory of $E_U$ where the input arities $X$ are discrete and the output arities $Y$ are arbitrary. Consider an $E$-algebra $A$. Then 
$$ 
	\xymatrix@=3pc{
		& A^Y \ar[r]^{A^p} &
		A^{Y_0}\ar@<0.5ex>[r]^{A^{p_0}}
		\ar@<-0.5ex>[r]_{A^{p_1}}& A^{Y_1}
	}
	$$
is a coreflexive inserter. Let $\omega$ be an $(X,Y)$-ary operation corresponding
to $f:FY\to FX$ and $\omega_0$ an $(X_0,Y_0)$-ary operation corresponding to $f_0$.
We get $(\omega_0)_A:A^{X_0}\to A^{Y_0}$. Since $f_0F(p_0)\leq f_0F(p_1)$, we have 
$A^{p_0}(\omega_0)_A\leq A^{p_1}(\omega_0)_A$. Thus there is a unique
$\tilde{f}:A^{X_0}\to A^Y$ such that $A^p\tilde{f}=\widetilde{f_0}$. Put
$\omega_A=\tilde{f}A^q$. This clearly makes $A$ an $E_U$-algebra and proves 
that the reduct functor $R:\Alg(E_U)\to\Alg(E)$ is an equivalence.

(2) Since $\Met$ is locally $\aleph_1$-presentable as a closed category, enriched monads preserving $\aleph_1$-directed colimits on $\Met$ correspond to $\aleph_1$-ary enriched equational theories over $\Met$. Analogously to (1), we can reduce them to
equational theories over $\Met$ without changing enriched algebras. 
Enriched $\aleph_1$-equational theories where the input arities $X$ are discrete metric spaces correspond to unconditional equational theories of \cite{MPP1}.
Following \ref{finpres}, every finite discrete metric space is finitely presentable in the enriched sense. Hence enriched monads given by finitary unconditional equational theories are finitary.

The basic equational theories of \cite{MPP1} are $\aleph_1$-ary enriched equational theories
where all $(X,1)$-ary operations symbols are induced by those having $X$ discrete.
This means that every $(X,1)$-ary operation symbol $\omega$ is equal to $\omega'x_u$ where
$u:X_0\to X$ is the identity on the underlying set of $X$ and $X_0$ is the discrete space on this set.
}
\end{exams}

\begin{nota}
{
\em
Let $\ck$ be an $\cm$-locally $\lambda$-generated $\cv$-category (see \cite{DR}). We say that a type $t$ over $\ck$ is $\lambda$-\textit{ary} if all arities $X$ and $Y$ of its operation symbols are $\lambda$-generated w.r.t. $\cm$ in the enriched sense. A theory is $\lambda$-ary if its type is $\lambda$-ary. 

We say that $\cv$ is \textit{$\cm$-locally $\lambda$-generated as  closed category} if it 
is $\cm$-locally $\lambda$-generated in the enriched sense, the tensor unit $I$ is $\lambda$-generated w.r.t. $\cm$ and the tensor product of two $\lambda$-generated objects w.r.t. $\cm$ is $\lambda$-generated w.r.t. $\cm$. 
}
\end{nota}

\begin{rem}\label{shift}
{
\em
Let $\ck$ be $\cm$-locally $\lambda$-generated as a closed category. Following 
\cite[Theorem 2.16]{DR}, there is a regular cardinal $\mu\geq\lambda$ such that $\ck$ is locally $\mu$-presentable and every $\lambda$-generated object w.r.t. $\cm$ is $\mu$-presentable. Since we can assume that $\lambda\triangleleft\mu$, \cite[Remark 2.15]{AR} implies that every $\mu$-presentable object $A$ is a retract of a $\mu$-small $\lambda$-directed colimit of $\lambda$-generated objects w.r.t. $\cm$. Since tensor product preserves colimits and retracts, $\ck$ is locally $\mu$-presentable as a closed category. This argument can be found in \cite[Proposition 2.4]{KL}.
}
\end{rem}

\begin{theo}\label{emonad3}
Let $\cv$ be $\cm$-locally $\lambda$-generated as a closed category, $\ck$ be an $\cm$-locally $\lambda$-generated $\cv$-category and $E$ be a $\lambda$-ary enriched equational theory over $\ck$. Then $\Alg(E)$ is an $\cm$-locally $\lambda$-generated $\cv$-category and $U_E$ is both $\cv$-monadic and a morphism of locally $\lambda$-generated $\cv$-categories.
\end{theo}
\begin{proof} 
The proof is analogous to that of \ref{monad3}. Following \ref{shift} and \ref{emonad0},
$\Alg(E)$ is locally $\mu$-presentable for some $\mu\geq\lambda$. The factorization system $(F(\ce),\cm')$ on $\Alg(E)$ is a $\cv$-factorization system because $\cm'$ is closed under cotensors (see \cite[Theorem 5.7]{LW}). Indeed, $\Alg(E)$ is cotensored as a locally presentable $\cv$-category and $U_E$ preserves cotensors because it has a left $\cv$-adjoint. Thus it suffices to use \cite[Theorem 4.17]{DR}.
\end{proof}

\begin{theo}\label{emonad4} 
Let $\cv$ be locally $\lambda$-generated as a closed category and $\ck$ be an  
$\cm$-locally $\lambda$-generated and locally $\mu$-presentable $\cv$-category for $\lambda\leq\mu$. Let $T$ be a $\mu$-ary $\cv$-monad on $\ck$ preserving $\cm$-morphisms and $\lambda$-directed colimits of $\cm$-morphisms. Then there is a $\lambda$-ary $\cv$-equational theory $E$ such that $\ck^T$ and $\Alg(E)$ are equivalent (as $\cv$-concrete categories over $\ck$).
\end{theo}
The proof is analogous to that of \ref{monad4}.

\section{More about metric monads}
At first, we summarize what we have shown in sections 3 and 4.
Since finitely generated metric spaces w.r.t. isometries coincide with finite metric spaces, our finitary equational theories are precisely equational theories whose operations have finite metric spaces as arities. The corresponding monads preserve $\aleph_1$-directed colimits and send
directed colimits of isometries to directed colimits. Conversely, if a monad preserves 
$\aleph_1$-directed colimits, isometries and sends directed colimits of isometries to directed colimits then it is given by a finitary equational theory. But monads given by finitary equational theories do not need to preserve isometries.  
Finite metric spaces also coincide with metric spaces finitely generated w.r.t. isometries in the enriched sense and $\Met$ is isometry-locally finitely generated as a closed category. Hence our finitary enriched equational theories are precisely enriched equational theories whose operations have finite metric spaces as arities. The corresponding enriched monads preserve $\aleph_1$-directed colimits and send directed colimits of isometries to directed colimits. This follows from \ref{emonad3} and \ref{emonad4}. Since the equational theory from \ref{counter} is enriched, enriched monads given by finitary enriched equational theories do not need to preserve isometries.

Now, we will look in more detail to finitary enriched equational theories over $\Met$ having
the input arities discrete.
Recall that a \textit{reflexive coequalizer} is a coequalizer of a reflexive pair, that is, parallel pair of split epimorphisms having a common splitting. Reflexive coequalizers are important
in universal algebra because they commute with finite products in $\Set$ (see \cite{ARV}).

\begin{propo}\label{commute1}
In $\Met$, finite products commute with reflexive coequalizers. 
\end{propo}
\begin{proof}
Following \ref{product}, it suffices to prove that finite products commute with reflexive coequalizers in $\PMet$. Like as in \ref{commute}, it suffices to show that the functor 
$$
X\times -:\PMet\to\PMet
$$ 
preserves reflexive coequalizers because the diagonal $D\to D\times D$ is final for the diagram scheme $\cd$ for a reflexive coequalizer.

Let $h:B\to C$ be a coequalizer of a reflexive pair $f,g:A\to B$ with $t:B\to A$. This means that $ft=gt=\id_B$. We have to show that $X\times h:X\times B\to X\times C$ is a coequalizer of the reflexive pair $X\times f,X\times g:X\times A\to X\times B$. Let $Z$ be its coequalizer. For $c_1,c_2\in C$, we have
$$
d(c_1,c_2)=\inf d(b_1,b_2)
$$
where the infimum is taken over all pairs $b_1,b_2\in B$ with $hb_1=c_1$ and $hb_2=c_2$. Since
the forgetful functor $U:\PMet\to\Set$ preserves all limits and colimits (see 
\cite[Examples 2.3(4)]{AR1}), the pseudometric spaces $X\times C$ and $Z$ have the same underlying set $X\times UC$. In the pseudometric space $X\times C$ we have distances
$$
d((x_1,c_1),(x_2,c_2))=\max\{d(x_1,x_2\},d(c_1,c_2)\}=\max\{d(x_1,x_2),\inf d(b_1,b_2)\}.
$$
while in $Z$
$$
d((x_1,c_1),(x_2,c_2))=\inf\max\{d(x_1,x_2),d(b_1,b_2)\}.
$$
It is easy to see that these distances are equal. In fact, if $d(x_1,x_2)\geq d(b_1,b_2)$ for some $b_1,b_2$ then the both distances are equal to $d(x_1,x_2)$. If $d(x_1,x_2)\leq d(b_1.b_2)$ for all $b_1,b_2$ then the both distances are equal to $\inf d(b_1,b_2)$.
\end{proof}

\begin{coro}\label{rcoeq} 
Let $A$ be a finite discrete metric space. Then its hom-functor 
$$
\Met(A,-):\Met\to\Met
$$ 
preserves reflexive coequalizers.
\end{coro}

Recall that a \textit{finitary unconditional equational theories} of \cite{MPP1} coincide with our finitary enriched equational theories where the input arities $X$ are discrete. 

\begin{coro}
Enriched monads on $\Met$ induced by finitary unconditional equational theories preserve filtered colimits and reflexive coequalizers.
\end{coro}
\begin{proof}
It follows from \ref{finpres} and \ref{rcoeq}.
\end{proof}

\textit{Sifted colimits} are those which commute in $\Set$ with finite products (see \cite{ARV}). Following \cite[Theorem 7.7]{ARV}, a functor $\Met\to\Met$ preserving filtered colimits and reflexive coequalizers preserve sifted colimits. 

\begin{lemma}\label{commute2}
A colimit commutes with finite products in $\Met$ iff is is sifted.
\end{lemma}
\begin{proof}
Following \ref{commute} and \ref{commute1}, filtered colimits and reflexive coequalizers commute
with finite products. For a general sifted colimit, it again suffices to show that the functor $X\times -:\Met\to\Met$ preserves sifted colimits because the diagonal $D\to D\times D$ is final for the diagram scheme $\cd$ for a sifted colimit (see 
\cite[Theorem 2.15]{ARV}). But this follows from \cite[Theorem 7.7]{ARV}.

The converse follows from the fact that discrete metric spaces form a coreflective full subcategory of $\Met$ which is closed under products. Hence every colimit commuting with finite products in metric spaces does it in discrete metric spaces, i.e., in $\Set$. Thus it is sifted.
\end{proof}

 \begin{rem}
{
\em
Discrete metric spaces do not form an enriched coreflective subcategory of $\Met$. Hence $\Met$ might
contain weighted colimits commuting with finite products.

In $\Pos$, the situation is analogous for discrete posets. And, indeed, reflexive coinserters
commute with finite products (see \cite{Bou}).
}
\end{rem}

Hence finitary unconditional equational theories yield monads preserving sifted colimits. 

\begin{exam}
{
\em
The Kantorovich monad on $\CMet$ is given by a quantitative equational theory for barycentric algebras (see \cite{FP}). Hence it preserves filtered colimits and reflexive coequalizers. The first fact was proved in \cite{AMMU}.
}
\end{exam}

But monads on $\Met$ preserving sifted colimits can also appear otherwise. 

\begin{lemma} \label{commute3}
Tensor product in a symmetric monoidal closed category $\cv$ commutes with filtered colimits and reflexive coequalizers.
\end{lemma}
\begin{proof}
See \cite[Lemma 2.3.2]{Re}.
\end{proof}

A monoid in a symmetric monoidal closed category $\cv$ is an object $M$ of $\cv$ equipped
with a morphism $m:M\otimes M\to M$ satisfying the associativity and unit axioms. Let
$\Mon(\cv)$ denote the category of monoids in $\cv$ and $V:\Mon(\cv)\to\cv$ the forgetful functor.
 
\begin{propo}\label{monoid}
Let $\cv$ be a symmetric monoidal closed category. Then $\Mon(\cv)$ has limits, filtered colimits and reflexive coequalizers and these limits and colimits are preserved by the forgetful functor.
\end{propo}
\begin{proof}
The existence and preservation of limits is evident. For filtered colimits and reflexive coequalizers, it follows \ref{commute3}.
\end{proof}

\begin{propo}[\cite{Po}]\label{pres}
Let $\cv$ be locally $\lambda$-presentable as a closed category. Then $\Mon(\cv)$ is locally $\lambda$-presentable and $V$ is monadic.  
\end{propo}
Following \ref{monoid}, the corresponding monad preserves sifted colimits. 

\begin{exam}\label{normed}
{
\em
(1) Let $\Norm$ be the category of generalized normed spaces (i.e., norm $\infty$ is allowed) and linear maps of norm $\leq 1$. Analogously as in \cite[Theorem 2.2]{R1}, we show that the forgetful functor $V:\Norm\to\Met$ is monadic. In fact, normed spaces are monoids in $\Met$ equipped with unary
operations $c\cdot-$, for $|c|\leq 1$, satisfying the appropriate axioms. The reason is that $+:VA\otimes VA\to VA$ and $c\cdot-:VA\to VA$ are nonexpanding. It does not seem that $\Norm$ can
be given by an equational theory over $\Met$.

(2) Let $\Ban$ be the category of (complex) Banach spaces and linear maps of norm 
$\leq 1$. Following \cite{R1}, the unit ball functor $U:\Ban\to\Met$ is monadic. Since $U$ is an enriched functor, we get an $\aleph_1$-enriched monad on $\Met$ (see 
\cite[Remark 4.2(3)]{R1}).  The left adjoint $F:\Met\to\Ban$ sends $1$ to the Banach space $\Bbb C$ of complex numbers. Since
$$
			\xymatrix@=3pc{
				1\ar[r]^{} \ar[d]_{} & 1\ar[d]^{} \\
				1 \ar[r]_{} & 2_\eps
			}
			$$
is an $\eps$-pushout for every $\eps>0$ and $\eps$-pushouts are weighted colimits (see \cite[Examples 4.1]{AR1}), $F2_\eps$ is an $\eps$-pushout
$$
			\xymatrix@=3pc{
				\Bbb C\ar[r]^{} \ar[d]_{} & \Bbb C\ar[d]^{} \\
				\Bbb C \ar[r]_{} & F2_\eps
			}
			$$
in $\Ban$. These $\eps$-pushouts are described in \cite[Lemma 6.5]{DR}.
}
\end{exam}

\begin{rem}
{
\em
Following \cite{Re}, every monad given by an operad on $\Met$ preserves sifted colimits.
This captures both monoids and generalized normed spaces.
}
\end{rem}

The category $\CMet$ of complete generalized metric spaces is a $\cv$-reflective full subcategory of $\Met$. It is locally $\aleph_1$-presentable (see 
\cite[Examples 2.3(2]{AR1}) and symmetric monoidal closed. The only $\aleph_0$-presentable object in $\CMet$ is the empty space (see \cite[Remark 2.7(1)]{AR1}). Since the reflector $\Met\to\CMet$ preserves finite products, finite products commute with filtered colimits and reflexive coequalizers in $\CMet$. Hence every finitary unconditional equational theory over $\CMet$ yields an enriched monad on $\CMet$ preserving filtered
colimits and reflexive coequalizers. Since \ref{commute2} is also valid in $\CMet$, it amounts to the preservation of sifted colimits.

\begin{exam}
{
\em
Consider the $\otimes$-theory from \ref{normed}(1) over $\CMet$ given by addition $+:Z\otimes Z\to Z$ and scalar multiplications $c\cdot-:Z\to Z$ for $c\in\Bbb C$, $|c|\leq 1$. Following \cite[Theorem 2.2]{R1}, $\Alg(\ct)=\Ban_\infty$ is the category of generalized Banach spaces with the forgetful functor $U_\ct:\Ban_\infty\to\CMet$. The corresponding monad on $\CMet$ preserves sifted colimits.
}
\end{exam}

Unlike $\Met$, $\CMet$ is not isometry-locally finitely generated (see \cite[Proposition 5.19]{AR1}). But $\Ban$ is isometry-locally finitely generated as a $\CMet$-category (see \cite[Remark 7.8]{AR1}). 
In the same time $\Ban$ is a monoidal category with the projective tensor product satisfying $\pa x\otimes_p y\pa=\pa x\pa\pa y\pa$ (see \cite[Examples 6.1.9h]{Bo}). Hence one can use our results to study monads on $\Ban$. 
\vskip 1mm

\noindent
{\bf Competing interests:} The author declare none.

\end{document}